\documentclass[11pt, a4paper]{article}

\usepackage{lineno}

\usepackage{graphicx}
\usepackage{ae}
\usepackage{amsmath}
\usepackage{amssymb}
\usepackage{latexsym}
\usepackage{url}
\usepackage{epsfig}

\usepackage{cite}
\usepackage{mathrsfs}
\usepackage{amsfonts}
\usepackage{amsthm}

\def\qed{\hfill$\Box$\vspace{12pt}}

\def\choose#1#2{\left (\!\!\begin{array}{c}#1\\#2\end{array}\!\!\right )}

\usepackage{color}

\input amssym.def
\input amssym.tex

\long\def\delete#1{}

\newcommand{\be}{\begin{equation}}
\newcommand{\ee}{\end{equation}}
\newcommand{\bea}{\begin{eqnarray}}
\newcommand{\eea}{\end{eqnarray}}
\newcommand{\bean}{\begin{eqnarray*}}
\newcommand{\eean}{\end{eqnarray*}}

\def\Cay{{\rm Cay}}

\def\Spec{{\rm Spec}}

\newtheorem{thm}{Theorem}[section]
\newtheorem{cor}[thm]{Corollary}
\newtheorem{lem}[thm]{Lemma}

\newtheorem{assump}[thm]{Assumption}
\newtheorem{rem}{Remark}
\numberwithin{equation}{section}

\title{Spectral properties of unitary Cayley graphs of finite commutative rings}

\author{Xiaogang Liu\, and\, Sanming Zhou
\\
{\small Department of Mathematics and Statistics}\\
{\small The University of Melbourne}\\
{\small Parkville, VIC 3010, Australia}\\
{\small xiaogliu@student.unimelb.edu.au, smzhou@ms.unimelb.edu.au}}

\date{}

\begin{document}

\openup 0.5\jot
\maketitle

\begin{abstract}
Let $R$ be a finite commutative ring. The unitary Cayley graph of $R$, denoted $G_R$, is the graph with vertex set $R$ and edge set $\left\{\{a,b\}:a,b\in R, a-b\in R^\times\right\}$, where $R^\times$ is the set of units of $R$. An $r$-regular graph is Ramanujan if the absolute value of every eigenvalue of it other than $\pm r$ is at most $2\sqrt{r-1}$. In this paper we give a necessary and sufficient condition for $G_R$ to be Ramanujan, and a necessary and sufficient condition for the complement of $G_R$ to be Ramanujan. We also determine the energy of the line graph of $G_R$, and compute the spectral moments of $G_R$ and its line graph.

\bigskip


\bigskip

\end{abstract}

\section{Introduction}

The {\em adjacency matrix} of a graph is the matrix with rows and columns indexed by its vertices such that the $(i,j)$-entry is equal to $1$ if vertices $i$ and $j$ are adjacent and $0$ otherwise. The \emph{eigenvalues} of a graph are eigenvalues of its adjacency matrix, and the \emph{spectrum} of a graph is the collection of its eigenvalues together with multiplicities. If $\lambda_1,\lambda_2,\ldots,\lambda_k$ are distinct eigenvalues of a graph $G$ and $m_1,m_2,\ldots,m_k$ the corresponding multiplicities, then we denote the spectrum of $G$ by
\begin{eqnarray*}
\Spec (G)=\left(\begin{array}{ccc}
\lambda_1 & \ldots  & \lambda_k \\
m_1      & \ldots  & m_k
\end{array}
\right).
\end{eqnarray*}

Let $R$ be a finite ring with unit element $1\neq0$, and let $R^\times$ denote its set of units. The \emph{unitary Cayley graph} \cite{kn:Fuchs04,kn:Fuchs05} of $R$, $G_R=\Cay(R,R^\times)$, is defined as the Cayley graph on the additive group of $R$ with respect to $R^\times$; that is, $G_R$ has vertex set $R$ such that $x, y \in R$ are adjacent if and only if $x-y\in R^\times$. It is evident that $G_R$ is a $|R^\times|$-regular undirected graph. Unitary Cayley graphs were introduced in \cite{kn:Fuchs04,kn:Fuchs05}, and their properties were investigated in \cite{kn:Akhtar09,kn:Kiani11,kn:Kiani12}, and \cite{kn:Droll10,kn:Ilic09,kn:Klotz07,kn:Ramaswamy09} in the special case when $R = \mathbb{Z}/n\mathbb{Z}$. For example, in \cite{kn:Klotz07} the chromatic number, clique number, independence number, diameter, vertex-connectivity and perfectness of $G_{\mathbb{Z}/n\mathbb{Z}}$ are determined.  In \cite{kn:Akhtar09}, the diameter, girth, eigenvalues, vertex-connectivity, edge-connectivity, chromatic number, chromatic index and automorphism group of $G_R$ are determined for an arbitrary finite commutative ring $R$, and all planar graphs and perfect graphs within this class are classified. The chromatic number, clique number and independence number of $G_R$ are also given in \cite{kn:Kiani12} along with other results. In \cite{kn:Droll10}, all unitary Cayley graphs $G_{\mathbb{Z}/n\mathbb{Z}}$ that are Ramanujan are classified.

A finite $r$-regular graph $G$ is called \emph{Ramanujan} \cite{HLW, kn:Murty03} if $\lambda(G) \leq 2\sqrt{r-1}$, where $\lambda(G)$ is the maximum in absolute value of an eigenvalue of $G$ other than $\pm r$. This notion arises from the well known Alon-Boppana bound (see \cite[Theorem 0.8.8]{DSV}), which asserts that $\liminf_{i \rightarrow \infty} \lambda(G_i) \ge 2\sqrt{r-1}$ for any family of finite, connected, $r$-regular graphs $\{G_i\}_{i \ge 1}$ with $|V(G_i)| \rightarrow \infty$ as $i \rightarrow \infty$. Over many years a great amount of work has been done on Ramanujan graphs with an emphasis on constructions of infinite families of Ramanujan $r$-regular graphs for a fixed integer $r$. The reader is referred to \cite{DSV} and two survey papers \cite{HLW, kn:Murty03} on Ramanujan graphs and related expander graphs.

The \emph{$k$-th spectral moment} of a graph $G$ with $n$ vertices and with eigenvalues $\lambda_1,\lambda_2,\ldots,\lambda_n$ is defined as
\[
s_k(G)=\sum_{i=1}^n \lambda_i^k,
\]
where $k \ge 0$ is an integer. The \emph{energy} of $G$ is defined as
$$
E(G) = \sum_{i=1}^n |\lambda_i|.
$$
Spectral moments are related to many combinatorial properties of graphs. For example, they play an important role in the proof by Lubotzky, Phillips and Sarnak \cite{LPS} of the Alon-Boppana bound. And the 4th spectral moment was used in \cite{RT} to give an upper bound on the energy of a bipartite graph.

The energy of a graph was introduced in \cite{kn:Gutman78} in the context of mathematical chemistry. Since then it has been studied extensively; see \cite{kn:Brualdiw, kn:Gutman01, kn:Gutman10, kn:Ilic09, kn:Kiani11, kn:Li10, kn:Ramane05, kn:Ramaswamy09, kn:Rojo11, kn:Rojo111, kn:Sander11} for examples. The energy of the unitary Cayley graph $G_{\mathbb{Z}/n\mathbb{Z}}$ was obtained in \cite{kn:Ilic09,kn:Ramaswamy09}, and that of its complement in \cite{kn:Ilic09}. This was generalized by D. Kiani et al. \cite{kn:Kiani11} to $G_R$ for an arbitrary finite commutative ring $R$.

The main results of the present paper are as follows. First, we give a necessary and sufficient condition (Theorems \ref{SXnThm1} and \ref{XnThm1}) for the unitary Cayley graph of any finite commutative ring to be Ramanujan, and a necessary and sufficient condition (Theorems \ref{ComSXnThm1} and \ref{ComXnThm222}) for the complement of such a graph to be Ramanujan. Second, we determine completely the energy of the line graph of $G_R$ for an arbitrary finite commutative ring $R$ (Theorem \ref{thmE}). Thirdly, we compute the spectral moments of $G_R$ and its line graph (Theorem \ref{thmSM}) for an arbitrary $R$.

In the special case when the ring considered is $\mathbb{Z}/n\mathbb{Z}$, Theorems \ref{SXnThm1} and \ref{XnThm1} recover (see Corollary \ref{XnCor1}) the classification \cite{kn:Droll10} of Ramanujan unitary Cayley graphs $G_{\mathbb{Z}/n\mathbb{Z}}$. We would like to point out that, although we obtain interesting infinite families of Ramanujan graphs in this way, they are not of fixed degrees. This is expected because it is known (see e.g.~\cite{kn:Murty05}) that for any given $r$ it is impossible to construct an infinite family of $r$-regular Cayley graphs on abelian groups which are all Ramanujan. As pointed out in \cite{kn:Droll10}, despite the fact that the theory of Ramanujan graphs is focused on infinite families of Ramanujan graphs with a fixed degree, constructions of infinite families of Ramanujan graphs of non-fixed degrees are also of some interest.

The rest of this paper is organised as follows. In the next section we collect some known results that will be used in subsequent sections. In Sections \ref{sec:ram} and \ref{sec:ram-comp}, we give characterisations of Ramanujan unitary Cayley graphs and Ramanujan complements of unitary Cayley graphs, respectively.
In Section \ref{sec:energy} we determine the energy of the unitary Cayley graph of any finite commutative ring. We finish the paper with a brief discussion on the spectral moments of unitary Cayley graphs.

\section{Preliminaries}
\label{sec:prel}

A \emph{local ring}  \cite{kn:Atiyah69} is a commutative ring with a unique maximal ideal. It is readily seen\cite{kn:Atiyah69,kn:Dummit03} that, if $R$ is a local ring with $M$ as its unique maximal ideal, then $R^\times=R\setminus M$. It is well known \cite{kn:Atiyah69,kn:Dummit03} that every finite commutative ring can be expressed as a direct product of finite local rings, and this decomposition is unique up to permutations of such local rings. Throughout the paper we assume the following:

\begin{assump}
\label{as:1}
$R=R_1\times R_2\times\cdots\times R_s$ is a finite commutative ring, where $R_i$ is a local ring with maximal ideal $M_i$ of order $m_i$, $1 \le i \le s$. We assume
\[|R_1|/m_1 \leq |R_2|/m_2 \leq \cdots \leq |R_s|/m_s.\]
\end{assump}

It is known \cite{kn:Akhtar09} that $G_R=\otimes_{i=1}^sG_{R_i}$ is the tensor product of $G_{R_1}, \ldots, G_{R_s}$. (The \emph{tensor product} $G\otimes H$ of two graphs $G$ and $H$ is the graph with vertex set $V(G)\times V(H)$, in which $(u,v)$ is adjacent to $(x,y)$ if and only if $u$ is adjacent to $x$ in $G$ and $v$ is adjacent to $y$ in $H$.) The degree of $G_R$ is equal to
\begin{equation}
\label{eq:basic}
|R^\times|=\prod_{i=1}^s(|R_i|-m_i)=\prod_{i=1}^s m_i \left((|R_{i}|/m_{i})-1\right)=|R|\prod_{i=1}^s\left(1-\frac{1}{|R_i|/m_i}\right).
\end{equation}
Define
$$
\lambda_C = (-1)^{|C|}\dfrac{|R^\times|}{\prod_{j\in C}(|R^\times_j|/m_j)}
$$
for every subset $C$ of $\{1,2,\ldots,s\}$.
In particular, $\lambda_{\emptyset} = |R^\times|$, and if $s=1$ then $\lambda_{\{1\}} = -m$, where $m$ is the order of the unique maximal ideal of $R$.

Proofs of our results rely on knowledge of the spectra of $G_R$, stated as follows.

\begin{lem}\label{ringspectrum}\emph{\cite{kn:Kiani11}}
The eigenvalues of $G_R$ are
\begin{itemize}
\item[\rm (a)] $\lambda_C$, repeated $\prod_{j\in C}|R^\times_j|/m_j$ times, where $C$ runs over all subsets of $\{1,2,\ldots,s\}$; and
\item[\rm (b)] $0$ with multiplicity $|R|-\prod_{i=1}^s\left(1+\dfrac{|R^\times_i|}{m_i}\right)$.
\end{itemize}
In particular, if $R$ is a finite local ring and $m$ is the order of its unique maximal ideal, then
\begin{eqnarray*}
\Spec (G_R)=\left(\begin{array}{ccc}
|R|-m & -m  & 0 \\
1          & \frac{|R|}{m}-1 & \frac{|R|}{m}(m-1)
\end{array}
\right)=\left(\begin{array}{ccc}
|R^\times| & -m  & 0 \\
1          & \frac{|R^\times|}{m} & \left(\frac{|R^\times|}{m}+1\right)(m-1)
\end{array}
\right).
\end{eqnarray*}
\end{lem}

\begin{rem}
\label{rem:mul}
{\em It may happen that $\lambda_C = \lambda_{C'}$ for distinct subsets $C, C'$ of $\{1,2,\ldots,s\}$. In fact, this occurs if and only if $|R_j|=2m_j$ for every $j \in (C \setminus C') \cup (C' \setminus C)$. Thus, in (a) above the multiplicity of the eigenvalue $\lambda_C$ may be greater than $\prod_{j\in C}|R^\times_j|/m_j$. For example, the multiplicity of the largest eigenvalue $|R^\times|$ of $G_R$ is equal to $\sum_C \left(\prod_{j\in C}|R^\times_j|/m_j\right)$, where the sum is running over all $C$ such that $|C|$ is even and $|R_j|=2m_j$ for every $j \in C$.}
\end{rem}

The following result was used in the proof \cite{kn:Kiani11} of Lemma \ref{ringspectrum}. It will be needed in our computing of the spectral moments of $G_R$.

\begin{lem}
\label{tensorspectrum}
\emph{\cite[Theorem 2.5.4]{kn:Cvetkovic10}}
Let $G$ and $H$ be graphs with eigenvalues $\lambda_1,\lambda_2,\ldots,\lambda_n$ and $\mu_1,\mu_2,\ldots,\mu_m$, respectively. Then the eigenvalues of $G\otimes H$ are $\lambda_i\mu_j$, $1\leq i\leq n, 1\leq j\leq m$.
\end{lem}

The \emph{complement} $\overline{G}$ of a graph $G$ is the graph with the same vertex set as $G$ such that two vertices are adjacent in $\overline{G}$ if and only if they are not adjacent in $G$.

\begin{lem}
\label{CXEigenvalues}
\emph{\cite{kn:Godsil01,kn:West00}}
Let $G$ be an $r$-regular graph with $n$ vertices. Then $G$ and $\overline{G}$ have the same eigenvectors, and their largest eigenvalues are $r$ and $n-r-1$ respectively. Moreover, if the eigenvalues of $G$ are $r,\lambda_2,\ldots,\lambda_n$, then the eigenvalues of $\overline{G}$ are $n-r-1,-1-\lambda_2,\ldots,-1-\lambda_n$.
\end{lem}

Lemmas \ref{ringspectrum} and \ref{CXEigenvalues} together imply the following result.

\begin{cor}\label{Corringspectrum}
The eigenvalues of $\overline{G}_R$ are
\begin{itemize}
\item[\rm (a)]  $|R|-1-|R^\times|$;
\item[\rm (b)] $-\lambda_C - 1$, repeated $\prod_{j\in C}|R^\times_j|/m_j$ times, where $C$ runs over all nonempty subsets of $\{1,2,\ldots,s\}$; and
\item[\rm (c)] $-1$ with multiplicity $|R|-\prod_{i=1}^s\left(1+\dfrac{|R^\times_i|}{m_i}\right)$.
\end{itemize}
In particular, if $R$ is a finite local ring and $m$ is the order of its unique maximal ideal, then
\begin{eqnarray*}
\Spec(\overline{G}_R)=\left(\begin{array}{cc}
m-1  & -1 \\
\frac{|R|}{m} & \frac{|R|}{m}(m-1)
\end{array}\right).
\end{eqnarray*}
\end{cor}

The \emph{line graph} $\mathcal{L}(G)$ of a graph $G$ is the graph with vertices the lines of $G$ such that two vertices are adjacent if and only if the corresponding lines have a common end-vertex.
It is well known \cite{kn:Sachs67} (see also \cite[Theorem 2.4.1]{kn:Cvetkovic10}) that, if an $r$-regular graph $G$ of order $n$ has eigenvalues $\lambda_1,\lambda_2,\ldots,\lambda_n$, then the eigenvalues of $\mathcal {L}(G)$ are $\lambda_i+r-2$, for $i=1,2,\ldots,n$, and $-2$ repeated $n(r-2)/2$ times. This together with Lemma \ref{ringspectrum} implies the following result.
\begin{cor}\label{lineeigenvalue}
The eigenvalues of $\mathcal {L}(G_R)$ are
\begin{itemize}
\item[\rm (a)] $\lambda_C+|R^\times|-2$, repeated $\prod_{j\in C}|R^\times_j|/m_j$ times, where $C$ runs over all subsets of $\{1,2,\ldots,s\}$;
\item[\rm (b)] $|R^\times|-2$ with multiplicity $|R|-\prod_{i=1}^s\left(1+\dfrac{|R^\times_i|}{m_i}\right)$; and
\item[\rm (c)] $-2$, repeated $|R|\left(|R^\times|-2\right)/2$ times.
\end{itemize}
In particular, if $R$ is a finite local ring and $m$ is the order of its unique maximal ideal, then
\begin{eqnarray*}
\Spec (\mathcal {L}(G_R))=\left(\begin{array}{cccc}
2|R^\times|-2 &|R^\times|-m-2  & |R^\times|-2 &-2\\
1          & \frac{|R^\times|}{m} & \left(\frac{|R^\times|}{m}+1\right)(m-1)&|R|(|R^\times|-2)/2
\end{array}
\right).
\end{eqnarray*}
\end{cor}

\begin{rem}
\label{rem:mu2}
{\em The multiplicity of the eigenvalue $-2$ of $\mathcal {L}(G_R)$ can be greater than $|R|(|R^\times|-2)/2$, and this happens if and only if there exists at least one subset $C$ of $\{1,2,\ldots,s\}$ with $|C|$ odd such that $\prod_{j\in C}|R^\times_j|/m_j=1$.}
\end{rem}

\begin{lem}\label{localmaximal}\emph{\cite[Proposition 2.1]{kn:Akhtar09}}
Let $R$ be a finite local ring and $m$ the order of its unique maximal ideal. Then there exists a prime $p$ such that $|R|$, $m$ and $|R|/m$ are all powers of $p$.
\end{lem}

Let
$$
n=p_1^{\alpha_1}p_2^{\alpha_2}\cdots p_s^{\alpha_s}
$$
be the canonical factorization of an integer $n$ into prime powers, where $p_1<p_2<\cdots<p_s$ are primes and each $\alpha_i\geq1$. The Euler's totient function is defined by
$\varphi(n)=n\prod\limits_{i=1}^s\left(1-(1/p_i)\right)$.

\begin{lem}\label{Inequality}\emph{\cite{kn:Ilic09}}
Let $n$ be as above. If $s\geq3$ or $s=2$ and $p_1>2$, then
\[2^{s-1}\varphi(n)>n.\]
\end{lem}

\section{Ramanujan unitary Cayley graphs}
\label{sec:ram}

In this section we characterise Ramanujan unitary Cayley graphs, as stated in the following two theorems.

\begin{thm}\label{SXnThm1}
Let $R$ be a finite local ring with maximal ideal $M$ of order $m$. Then $G_R$ is Ramanujan if and only if one of the following holds:
\begin{itemize}
\item[\rm (a)] $|R|=2m$;
\item[\rm (b)] $|R|\geq\left(\dfrac{m}{2}+1\right)^2$ \emph{and} $m \neq 2$;
\end{itemize}
\end{thm}

\begin{proof}
Recall that $G_R$ is regular of degree $|R^\times|$, and that $R^{\times} = R \setminus M$ as $R$ is local. Note that $|R| \ge 2m$. If $|R|=2m$, then by Lemma \ref{ringspectrum},
\begin{eqnarray*}
\Spec (G_R)
=\left(\begin{array}{ccc}
|R^\times| & -|R^\times|  & 0 \\
1   & 1 & |R|-2
\end{array}
\right)
\end{eqnarray*}
and so $G_R$ is Ramanujan.

Assume $|R|>2m$. Then Lemma \ref{ringspectrum} implies that $G_R$ is Ramanujan if and only if
$m \leq 2\sqrt{|R^\times|-1}$, or equivalently, $|R|\geq\left((m/2)+1\right)^2$.
Note that we always have $\left((m/2)+1\right)^2 \ge 2m$, with equality precisely when $m=2$. In the case when $m=2$, it is well known \cite{kn:Ganesan64} that either $R\cong\mathbb{Z}_4$ or $R\cong\mathbb{Z}_2[X]/(X^2)$. In either case we have $|R|=4$, which contradicts $|R|>4$. \qed \end{proof}

\begin{thm}\label{XnThm1}
Let $R$ be as in Assumption \ref{as:1} with $s \ge 2$. Then $G_R$ is Ramanujan if and only if $R$ satisfies one of the following conditions:
\begin{itemize}
\item[\rm (a)] $R_i/M_i\cong  \mathbb{F}_2$ for $i = 1, \ldots, s$;

\item[\rm (b)] $R_i \cong  \mathbb{F}_2$ for $i = 1, \ldots, s-3$, and  $R_i \cong  \mathbb{F}_3$ for $i = s-2, s-1, s$;

\item[\rm (c)] $R_i \cong  \mathbb{F}_2$ for $i = 1, \ldots, s-3$, $R_i \cong  \mathbb{F}_3$ for $i = s-2, s-1$, and $R_s\cong\mathbb{F}_4$;

\item[\rm (d)] $R_i \cong  \mathbb{F}_2$ for $i = 1, \ldots, s-3$, and  $R_i \cong  \mathbb{F}_4$ for $i = s-2, s-1, s$;

\item[\rm (e)] $R_i \cong  \mathbb{F}_2$ for $i = 1, \ldots, s-2$, $R_{s-1}\cong  \mathbb{F}_3$, and $R_{s}\cong  \mathbb{Z}_9 $ or $\mathbb{Z}_3[X]/(X^3)$;

\item[\rm (f)] $R_1\cong\mathbb{Z}_4$ or $\mathbb{Z}_2[X]/(X^2)$, $R_i \cong  \mathbb{F}_2$ for $i = 2, \ldots, s-2$, and $R_{s-1}\cong  \mathbb{F}_{q_1}$, $R_{s}\cong  \mathbb{F}_{q_2}$ for some prime powers $q_1, q_2 \ge 3$ such that
\begin{equation}
\label{eq:f}
q_1 \le q_2\leq q_1+\sqrt{(q_1-2)q_1};
\end{equation}

\item[\rm (g)] $R_i \cong  \mathbb{F}_2$ for $i = 1, \ldots, s-2$, and $R_{s-1}\cong  \mathbb{F}_{q_1}$, $R_{s}\cong  \mathbb{F}_{q_2}$ for some prime powers $q_1, q_2 \geq3$ such that
\begin{equation}
\label{eq:g}
q_1 \le q_2\leq2\left(q_1+\sqrt{(q_1-2)q_1}\right)-1;
\end{equation}

\item[\rm (h)] $R_i/M_i\cong  \mathbb{F}_2$ for $i = 1, \ldots, s-1$, $R_{s}/M_{s}\cong  \mathbb{F}_q$ for some prime power $q\geq3$, and
\begin{equation}
\label{eq:h}
\prod_{i=1}^s m_i \leq2\left(q-1+\sqrt{(q-2)q}\right).
\end{equation}
\end{itemize}
\end{thm}

\begin{proof}
Note that $|R_i|/m_i \geq 2$, $1 \le i \le s$, and the degree of $G_R$ is given in (\ref{eq:basic}).

\smallskip
\noindent\emph{Case 1:} $|R_1|/m_1=|R_2|/m_2=\cdots=|R_s|/m_s=2$. In this case all non-zero eigenvalues of $G_R$ have absolute value $|\lambda_C|=|R^\times|=|R|/2^s$,
which implies that $G_R$ is Ramanujan, as claimed in (a).

\smallskip
\noindent\emph{Case 2:} There exists at least one $j\in\{1,2,\ldots,s\}$ such that $|R_j|/m_j>2$. Let $t+1$ be the largest $j$ such that this occurs, so that $0 \le t < s$ and
\begin{equation}
\label{ratiosequence}
2=|R_1|/m_1=\cdots=|R_t|/m_t<|R_{t+1}|/m_{t+1} \leq \cdots \leq |R_{s}|/m_{s},
\end{equation}
where by convention if $t=0$ then all $|R_i|/m_i>2$. Since $G_R$ is $|R^\times|$-regular, it is Ramanujan if and only if $|\lambda_C| \leq2\sqrt{|R^\times|-1}$ for all eigenvalues $\lambda_C \neq\pm|R^\times|$ of $G_R$.
Note that $|\lambda_C| < |R^\times|$ is maximized if and only if $\prod_{j\in C}\left(|R_j|/m_j-1\right)$ is minimized.
If $C \subseteq \{1, \ldots, t\}$, then $|\lambda_C| = |R^\times|$. If $C \cap \{t+1, \ldots, s\} \ne \emptyset$, then
$|\lambda_C| = |\lambda_{C \cap \{t+1, \ldots, s\}}| \le |\lambda_{\{t+1\}}|$.
Thus $G_R$ is Ramanujan if and only if $|\lambda_{\{t+1\}}| \leq2\sqrt{|R^\times|-1}$, that is,
\begin{equation}\label{Rcondition}
\dfrac{|R^\times|}{(|R_{t+1}|/m_{t+1})-1}\leq2\sqrt{|R^\times|-1}.
\end{equation}
Since $2\sqrt{|R^\times|-1} < 2\sqrt{|R^\times|}$, this condition is not satisfied unless
\begin{eqnarray}\label{RNescondition}
|R^\times| <4\left((|R_{t+1}|/m_{t+1})-1\right)^2.
\end{eqnarray}
In particular, if $s\geq t+4$, then $|R^\times| \ge \prod_{i=t+1}^s\left((|R_{i}|/m_{i})-1\right)$ $\geq$ $4\left((|R_{t+1}|/m_{t+1})-1\right)^2$ by (\ref{eq:basic}) and (\ref{ratiosequence}), and hence $G_R$ is not Ramanujan. It remains to consider the case where $s-3 \le t < s$.

\smallskip
\noindent\emph{Case 2.1:} $s=t+3$. In view of (\ref{eq:basic}), in this case (\ref{RNescondition}) is mounted to
$$
\prod_{i=1}^{t+3} m_i \left((|R_{t+2}|/m_{t+2})-1\right)\left((|R_{t+3}|/m_{t+3})-1\right) < 4\left((|R_{t+1}|/m_{t+1})-1\right).
$$
Note that if $\prod_{i=1}^{t+3} m_i \geq 2$ or $|R_{t+3}|/m_{t+3} \geq 5$, then this condition is not satisfied and hence $G_R$ is not Ramanujan. Now we assume $\prod_{i=1}^{t+3} m_i=1$ and $|R_{t+3}|/m_{t+3}\leq4$. Then one of the following occurs: (i) $|R_{t+1}|=|R_{t+2}|=|R_{t+3}|=3$; (ii) $|R_{t+1}|=|R_{t+2}|=3$ and $|R_{t+3}|=4$; (iii) $|R_{t+1}|=|R_{t+2}|=|R_{t+3}|=4$; (iv) $|R_{t+1}|=3$ and $|R_{t+2}|=|R_{t+3}|=4$. In cases (i)-(iii), (\ref{Rcondition}) is satisfied and so $G_R$ is Ramanujan as claimed in (b), (c) and (d); whilst in (iv), (\ref{Rcondition}) is not satisfied and so $G_R$ is not Ramanujan.

\smallskip
\noindent\emph{Case 2.2:} $s=t+2$. In this case (\ref{RNescondition}) is mounted to
$$
\prod_{i=1}^{t+2} m_i \left((|R_{t+2}|/m_{t+2})-1\right) < 4\left((|R_{t+1}|/m_{t+1})-1\right).
$$
Thus, if $\prod_{i=1}^{t+2} m_i \geq4$, then $G_R$ is not Ramanujan. Assume $\prod_{i=1}^{t+2} m_i \leq3$ in the sequel.

\smallskip
\noindent\emph{Case 2.2.1:} $\prod_{i=1}^{t+2} m_i=3$. Then (\ref{Rcondition}) is mounted to
\[
3((|R_{t+2}|/m_{t+2})-1)\leq2\sqrt{3((|R_{t+1}|/m_{t+1})-1)((|R_{t+2}|/m_{t+2})-1)-1},
\]
which is equivalent to
\begin{eqnarray}\label{M3Case221}
|R_{t+2}|/m_{t+2}\leq\frac{2}{3}\left((|R_{t+1}|/m_{t+1})+\sqrt{((|R_{t+1}|/m_{t+1})-2)(|R_{t+1}|/m_{t+1})}\right)+\frac{1}{3}.
\end{eqnarray}
Note that by Lemma \ref{localmaximal} and (\ref{ratiosequence}), we have $m_i=1$, $1 \le i \le t$, and $(m_{t+1}, m_{t+2})=(1,3)$ or $(3,1)$. It is well known \cite{kn:Ganesan64} that $\mathbb{Z}_9$ and $\mathbb{Z}_3[X]/(X^3)$ are the only local rings whose unique maximal ideal has exactly three elements. Thus, one of the following holds: (i) $R_{t+1}\cong  \mathbb{F}_q$, and $R_{t+2}\cong  \mathbb{Z}_9$ or $\mathbb{Z}_3[X]/(X^3)$; (ii) $R_{t+1}\cong  \mathbb{Z}_9$ or $\mathbb{Z}_3[X]/(X^3)$, and $R_{t+2}\cong  \mathbb{F}_q$, where $q\geq3$ is a prime power. In case (i), by (\ref{ratiosequence}), we have $q=3$, and as stated in (e), $G_R$ is Ramanujan since (\ref{M3Case221}) is satisfied. In case (ii), (\ref{M3Case221}) is satisfied only when $q=3$, and in this case $G_R$ is Ramanujan as stated in (e).

\smallskip
\noindent\emph{Case 2.2.2:} $\prod_{i=1}^{t+2}m_i=2$. Then (\ref{Rcondition}) is mounted to
\[(|R_{t+2}|/m_{t+2})-1\leq\sqrt{2((|R_{t+1}|/m_{t+1})-1)((|R_{t+2}|/m_{t+2})-1)-1},\]
which is equivalent to
\begin{equation}
\label{eq:added}
|R_{t+2}|/m_{t+2}\leq (|R_{t+1}|/m_{t+1})+\sqrt{((|R_{t+1}|/m_{t+1})-2)(|R_{t+1}|/m_{t+1})}.
\end{equation}
Similar to Case 2.2.1, we have, say, $m_1=2$ and $m_i=1$, $2 \le i \le s$ (note that in this case $t\geq1$). Since $\mathbb{Z}_4$ and $\mathbb{Z}_2[X]/(X^2)$ are the only local rings whose unique maximal ideal has order two \cite{kn:Ganesan64}, we have $R_1\cong\mathbb{Z}_4$ or $\mathbb{Z}_2[X]/(X^2)$, $R_i \cong  \mathbb{F}_2$, $2 \le i \le t$, $R_{t+1}\cong  \mathbb{F}_{q_1}$ and $R_{t+2}\cong  \mathbb{F}_{q_2}$, where $q_1, q_2 \ge 3$ are prime powers. By (\ref{eq:added}), if $q_1 \le q_2 \leq q_1+\sqrt{(q_1-2)q_1}$, then $G_R$ is Ramanujan as claimed in (f).

\smallskip
\noindent\emph{Case 2.2.3:} $\prod_{i=1}^{t+2} m_i = 1$. Then all $R_i$ are finite fields. Thus, $R_i \cong  \mathbb{F}_2$, $1 \le i \le t$, $R_{t+1}\cong  \mathbb{F}_{q_1}$ and $R_{t+2}\cong  \mathbb{F}_{q_2}$, where $q_2 \ge q_1 \ge 3$ are prime powers. By (\ref{Rcondition}), $G_R$ is Ramanujan if and only if $q_2-1\leq2\sqrt{(q_1-1)(q_2-1)-1}$, which is equivalent to
$q_2\leq2\left(q_1+\sqrt{(q_1-2)q_1}\right)-1$, yielding (g).

\smallskip
\noindent\emph{Case 2.3:}  $s=t+1$. Then $R_i/M_i\cong  \mathbb{F}_2$, $1 \le i \le t$, $R_{t+1}/M_{t+1}\cong  \mathbb{F}_q$, and $|R^\times|=\prod_{i=1}^{t+1} m_i (q-1)$, where $q \ge 3$ is a prime power. Thus (\ref{Rcondition}) is mounted to $\prod_{i=1}^{t+1} m_i \leq 2\left(q-1+\sqrt{(q-2)q}\right)$, and in this case $G_R$ is Ramanujan as stated in (h).
\qed
\end{proof}

\begin{rem}
\label{rem:3}
{\em (a) It is known \cite{kn:Akhtar09} that, if $R$ is a local ring with maximal ideal $M$, then $G_R$ is a complete multipartite graph whose partite sets are the cosets of $M$ in $R$ (in particular, $G_R$ is a complete graph when $|M|=1$). Thus, since $G_R=\otimes_{i=1}^sG_{R_i}$, in each case of Theorem \ref{XnThm1}, $G_R$ is a tensor product whose factor graphs are complete or complete multipartite.

(b) It is well known \cite{kn:Murty03} that for an $r$-regular graph $G$ the multiplicity of $r$ as an eigenvalue is equal to the number of connected components of $G$. Thus $G_R$ in Theorem \ref{SXnThm1} is always connected. In Theorem \ref{XnThm1}, $G_R$ is connected if and only if there is at most one factor $R_i$ such that $R_i/M_i\cong\mathbb{F}_2$. In particular, Theorem \ref{XnThm1} gives four infinite families of connected Ramanujan graphs: (i) $C_4 \otimes K_{q_1} \otimes K_{q_2}$, with $q_1, q_2 \ge 3$ prime powers satisfying (\ref{eq:f}); (ii)-(iii) $K_2 \otimes K_{q_1} \otimes K_{q_2}$ and $K_{q_1} \otimes K_{q_2}$, with $q_1, q_2 \ge 3$ prime powers satisfying (\ref{eq:g}); (iv) $K_{m_1, m_1} \otimes K_{m_2, \ldots, m_2}$, where $K_{m_2, \ldots, m_2}$ has $q$ parts for a prime power $q$, and $m_1 m_2 \leq 2(q-1+\sqrt{(q-2)q})$.}
\end{rem}

Let $n=p_1^{\alpha_1}p_2^{\alpha_2}\cdots p_s^{\alpha_s}$ be the canonical factorisation of an integer $n$, where $p_1<p_2<\cdots<p_s$ are primes. It is well known (see e.g.~\cite{kn:Dummit03}) that
$$
\mathbb{Z}/n\mathbb{Z}\cong(\mathbb{Z}/p_1^{\alpha_1}\mathbb{Z})\times(\mathbb{Z}/p_2^{\alpha_2}\mathbb{Z})\times\cdots\times(\mathbb{Z}/p_s^{\alpha_s}\mathbb{Z}),
$$
where each $R_i=\mathbb{Z}/p_i^{\alpha_i}\mathbb{Z}$ is a local ring with unique maximal ideal $M_i = (p_i)/(p_i^{\alpha_i})$ of order $m_i=|M_i|=p_i^{\alpha_i-1}$. In the special case where $R = \mathbb{Z}/n\mathbb{Z}$, Theorems \ref{SXnThm1} and \ref{XnThm1} together imply the following known result.

\begin{cor}
\label{XnCor1}
\emph{\cite[Theorem 1.2]{kn:Droll10}}
Let $n=p_1^{\alpha_1}p_2^{\alpha_2}\cdots p_s^{\alpha_s}$ be as above. Then $G_{\mathbb{Z}/n\mathbb{Z}}$ is Ramanujan if and only if one of the following holds:
\begin{itemize}
\item[\rm (a)] $n=2^{\alpha_1}$  with $\alpha_1\geq1$;
\item[\rm (b)] $n=p_1^{\alpha_1}$ with $p_1$ odd and $\alpha_1=1,2$;
\item[\rm (c)] $n=4p_2p_3$ with $p_2 < p_3\leq 2p_2-3$;
\item[\rm (d)] $n=p_1p_2$ with $3 \le p_1 < p_2\leq 4p_1-5$, or $n=2p_2p_3$ with $3 \le p_2 < p_3\leq 4p_2-5$;
\item[\rm (e)] $n=2p_2^2,~4p_2^2$ with $p_2$ odd, or $n=2^{\alpha_1}p_2$ with $p_2>2^{\alpha_1-3}+1$.
\end{itemize}
\end{cor}
\begin{proof}
If $n=2^{\alpha_1}$, then $R_1=\mathbb{Z}/2^{\alpha_1}\mathbb{Z}$, $|R_1|/m_1=2$, and $G_{\mathbb{Z}/n\mathbb{Z}}$ is Ramanujan by (a) of Theorem \ref{SXnThm1}.

If $n=p_1^{\alpha_1}$ with $p_1$ odd, then $R_1=\mathbb{Z}/p^{\alpha_1}\mathbb{Z}$ and $m_1=p^{\alpha_1-1}$. In this case, by (b) of Theorem \ref{SXnThm1} $G_R$ is Ramanujan if and only if $p^{\alpha_1}\geq\left((p^{\alpha_1-1}/2)+1\right)^2$, which holds if and only if $\alpha_1=1$ or $2$.

Now we assume $s \ge 2$. It can be easily verified that none of (a)-(e) in Theorem \ref{XnThm1} can occur. In (f) of Theorem \ref{XnThm1}, we have $s=3$, $\prod_{i=1}^3 m_i = 2$, and hence $p_1=2$ and $n=4p_2p_3$. The second inequality in (\ref{eq:f}) is equivalent to $p_3\leq2p_2-3$, yielding (c).

In (g) of Theorem \ref{XnThm1}, we have $s=2$ or $3$, and $\prod_{i=1}^s m_i = 1$, implying $\alpha_1= \cdots = \alpha_s=1$.  If $s=2$, then $n=p_1p_2$ with $p_1 \geq3$, and (\ref{eq:g}) is equivalent to $p_2\leq 4p_1-5$, leading to the first possibility in (d). Similarly, if $s=3$, then $n=2p_2p_3$ with $p_2\geq3$, leading to the second possibility in (d).

In (h) of Theorem \ref{XnThm1}, we have $s=2$ and $n = 2^{\alpha_1} p_2^{\alpha_2}$ with $p_2\geq3$. The inequality (\ref{eq:h}) is mounted to $2^{\alpha_1-1}p_2^{\alpha_2-1}\leq2\left(p_2-1+\sqrt{p_2(p_2-2)}\right)$, which holds only if
$2^{\alpha_1-1}p_2^{\alpha_2-1}<4p_2-4$.
This latter inequality holds only when $\alpha_2 = 1$ and $p_2 > 2^{\alpha_1 - 3} + 1$, $(\alpha_1, \alpha_2) = (1, 2)$, or $(\alpha_1, \alpha_2) = (2, 2)$, and in each of these cases (\ref{eq:h}) is satisfied, leading to (e).
\qed \end{proof}

\section{Ramanujan complements of unitary Cayley graphs}
\label{sec:ram-comp}

Corollary \ref{Corringspectrum} implies the following result.

\begin{thm}\label{ComSXnThm1}
Let $R$ be a finite local ring. Then $\overline{G}_R$ is Ramanujan.
\end{thm}

In the general case where $s \ge 2$, we obtain the following:

\begin{thm}\label{ComXnThm222}
Let $R$ be as in Assumption \ref{as:1} with $s \ge 2$. Then $\overline{G}_R$ is Ramanujan if and only if $R$ satisfies one of the following conditions:
\begin{itemize}
\item[\rm (a)] $|R_i|/m_i=2$, $1 \le i \le s$, and
\begin{equation}
\label{eq:a}
\prod_{i=1}^s m_i \leq 2^{s+1}-3+2\sqrt{2^s(2^s-3)};
\end{equation}
\item[\rm (b)] $2=|R_1|/m_1=\cdots=|R_t|/m_t < |R_{t+1}|/m_{t+1}$ for some $t$ with $2 \le t < s$, and
$$
|R^\times|\leq2\sqrt{|R|}-3;
$$
\item[\rm (c)] $2=|R_1|/m_1 < |R_2|/m_2$ and
$$
|R^\times|\leq2\sqrt{|R|-2}-1;
$$
\item[\rm (d)] $3\leq|R_1|/m_1$ and
\begin{equation}
\label{eq:d}
\frac{|R^\times|}{(|R_1|/m_1)-1}\leq-\left(2(|R_1|/m_1)-3\right)+\sqrt{\left(2(|R_1|/m_1)-3\right)^2+(4|R|-9)}.
\end{equation}
\end{itemize}
\end{thm}

\begin{proof}
Note that $|R_i|/m_i \geq 2$, $1 \le i \le s$, and by (\ref{eq:basic}) the degree of $\overline{G}_R$ is equal to
$$
|R|-1-|R^\times|=|R|-1-\prod_{i=1}^s(|R_i|-m_i)=|R|\left(1-\prod_{i=1}^s\left(1-\frac{1}{|R_i|/m_i}\right)\right)-1.
$$
Denote by $\mu$ the maximum absolute value of the eigenvalues of $\overline{G}_R$ other than $|R|-1-|R^\times|$.

\smallskip
\noindent\emph{Case 1:} $|R_1|/m_1=|R_2|/m_2=\cdots=|R_s|/m_s=2$. Since $s \ge 2$, we have $|R^\times|+1 = \prod_{i=1}^s m_i + 1 < |R|-1-|R^\times|$. Thus, by Corollary \ref{Corringspectrum}, $\mu = |-\lambda_{\{1,2\}}-1| = |R^\times|+1=\prod_{i=1}^s m_i+1$. Hence $\overline{G}_R$ is Ramanujan if and only if $\prod_{i=1}^s m_i+1\leq2\sqrt{(2^s-1)\prod_{i=1}^s m_i-2}$, which is equivalent to (\ref{eq:a}), leading to (a).

\smallskip
\noindent\emph{Case 2:} $2=|R_1|/m_1=\cdots=|R_t|/m_t<|R_{t+1}|/m_{t+1}$ for some $t$ with $2 \le t \le s$. In this case one can verify that $|R^\times|+1 < |R|-1-|R^\times|$ and so $\mu = |-\lambda_{\{1,2\}}-1| = |R^\times|+1$ by Corollary \ref{Corringspectrum}. Thus $\overline{G}_R$ is Ramanujan if and only if $|R^\times|+1\leq2\sqrt{|R|-2-|R^\times|}$, that is, $|R^\times|\leq2\sqrt{|R|}-3$, leading to (b).

\smallskip
\noindent\emph{Case 3:} $2=|R_1|/m_1<|R_2|/m_2$. In this case, we have $|R^\times|-1 < |R|-1-|R^\times|$ and $\mu = |-\lambda_{\{1\}}-1| = |R^\times|-1$. Thus, $\overline{G}_R$ is Ramanujan if and only if $|R^\times|-1\leq2\sqrt{|R|-2-|R^\times|}$, leading to (c).

\smallskip
\noindent\emph{Case 4:} $3\leq|R_1|/m_1$. In this case
$\mu = |-\lambda_{\{1\}}-1| =\frac{|R^\times|}{(|R_1|/m_1)-1}-1$ ($< |R|-1-|R^\times|$).
Hence $\overline{G}_R$ is Ramanujan if and only if
$\frac{|R^\times|}{(|R_1|/m_1)-1}-1\leq2\sqrt{|R|-2-|R^\times|}$, which leads to (d).
\qed \end{proof}

Applying Theorems \ref{ComSXnThm1} and \ref{ComXnThm222} to $\mathbb{Z}/n\mathbb{Z}$, we obtain the following corollary.

\begin{cor}
\label{cor:comp}
Let $n \ge 2$ be an integer. Then $\overline{G}_{\mathbb{Z}/n\mathbb{Z}}$ is Ramanujan if and only if $n$ is of one of the following forms:
\begin{itemize}
\item[\rm (a)] $n=p^{\alpha}$ with $p$ a prime and $\alpha \geq 1$;
\item[\rm (b)] $n=2\cdot3\cdot5$, $2\cdot3$, $2\cdot3^2$, $2\cdot5$, $2^2\cdot3$, $2^3\cdot3$, $3\cdot5$, $3\cdot7$ or $5\cdot7$.
\end{itemize}
\end{cor}

\begin{proof}
We use the notation before Corollary \ref{XnCor1}. Then $|R^\times|=\prod_{i=1}^sp_i^{\alpha_i-1}(p_i-1)=\varphi(n)$ and $|R_i|/m_i=p_i$ for each $i$.

\smallskip
\noindent\emph{Case 1:} $n=p^{\alpha}$. Then $\overline{G}_{\mathbb{Z}/n\mathbb{Z}} \cong \overline{K}_n$ is Ramanujan by Theorem \ref{ComSXnThm1}.

\smallskip
\noindent\emph{Case 2:} $n=2^{\alpha_1}p_2^{\alpha_2}\cdots p_s^{\alpha_s}$, where $s \ge 1$. Then case (c) of Theorem \ref{ComXnThm222} applies, and $\overline{G}_{\mathbb{Z}/n\mathbb{Z}}$ is Ramanujan if and only if $\varphi(n)\leq2\sqrt{n-2}-1$. This condition is satisfied only if $\varphi(n)^2/n<4$. In particular, if $s\geq4$, then by Lemma \ref{Inequality}, $\varphi(n)^2/n > \varphi(n)/2^{s-1} > 4$ and so $\overline{G}_{\mathbb{Z}/n\mathbb{Z}}$ is not Ramanujan. Assume $s\leq3$ in the sequel.

\smallskip
\noindent\emph{Case 2.1:} $s=3$. Since $(p_i - 1)^2 > p_i (p_i - 2)$, if $\alpha_1\geq3$, $\alpha_2\geq2$ or $\alpha_3\geq2$, then
$\varphi(n)^2/n = 2^{\alpha_1-2}p_2^{\alpha_3-2}p_3^{\alpha_3-2}(p_2-1)^2(p_3-1)^2>2^{\alpha_1-2}p_2^{\alpha_2-1}p_3^{\alpha_3-1}(p_2-2)(p_3-2)\geq4$, and so $\overline{G}_{\mathbb{Z}/n\mathbb{Z}}$ is not Ramanujan. It remains to consider the case where $n = 2p_2p_3$ or $4p_2p_3$.

If $n=2p_2p_3$, then $\varphi(n)^2/n = (p_2-1)^2(p_3-1)^2/2p_2p_3 > (p_2-2)(p_3-2)/2 \geq 4$ if
$(p_2, p_3) \ne (3, 5)$ or $(3, 7)$. Thus, unless $(p_2, p_3) = (3, 5)$ or $(3, 7)$, $\overline{G}_{\mathbb{Z}/n\mathbb{Z}}$ is not Ramanujan. It is easy to see that if $(p_2, p_3) = (3, 5)$, then $\varphi(n)\leq2\sqrt{n-2}-1$ and so $\overline{G}_{\mathbb{Z}/n\mathbb{Z}}$ is Ramanujan, whilst if $(p_2, p_3) = (3, 7)$, then $\varphi(n) > 2\sqrt{n-2}-1$ and so $\overline{G}_{\mathbb{Z}/n\mathbb{Z}}$ is not Ramanujan.

If $n=4p_2p_3$, then $\varphi(n)^2/n = (p_2-1)^2(p_3-1)^2/p_2p_3 > (p_2-2)(p_3-2) \geq 4$ unless $(p_2, p_3)=(3, 5)$. Moreover, if $(p_2, p_3)=(3, 5)$, then $\varphi(n) > 2\sqrt{n-2}-1$. Hence $\overline{G}_{\mathbb{Z}/n\mathbb{Z}}$ is not Ramanujan when $n=4p_2p_3$.

\smallskip
\noindent\emph{Case 2.2:} $s=2$. In this case $\varphi(n)^2/n = 2^{\alpha_1-2}p_2^{\alpha_2-2}(p_2-1)^2>2^{\alpha_1-2}p_2^{\alpha_2-1}(p_2-2)$. From this one can see that $\varphi(n)^2/n \geq 4$ if $\alpha_1\geq4$, $\alpha_2\geq3$ or $p_2\geq7$, or if $n=2\cdot5^2$, $2^2\cdot5^2$, $2^2\cdot3^2$, $2^3\cdot3^2$, $2^3\cdot5$ or $2^3\cdot5^2$. Thus, unless $n=2\cdot3$,  $2\cdot3^2$, $2\cdot5$, $2^2\cdot3$, $2^2\cdot5$ or $2^3\cdot3$, $\overline{G}_{\mathbb{Z}/n\mathbb{Z}}$ is not Ramanujan. It can be verified that in all these exceptional cases, except when $n= 2^2\cdot5$, $\overline{G}_{\mathbb{Z}/n\mathbb{Z}}$ is Ramanujan.

\smallskip
\noindent\emph{Case 3:} $n=p_1^{\alpha_1}p_2^{\alpha_2}\cdots p_s^{\alpha_s}$ with $p_1\geq3$. Then case (d) of Theorem \ref{ComXnThm222} applies, and by (\ref{eq:d}), $\overline{G}_{\mathbb{Z}/n\mathbb{Z}}$ is Ramanujan if and only if
\begin{equation}
\label{CXRamanujan22}
\frac{\varphi(n)}{p_1-1} \leq -(2p_1-3)+\sqrt{(2p_1-3)^2+(4n-9)}.
\end{equation}
Note that this condition is not satisfied unless $\varphi(n)^2/n < 4(p_1-1)^2$.
In particular, if $s\geq4$, then by Lemma \ref{Inequality}, $\varphi(n)^2/n > \varphi(n)/2^{s-1} > 4(p_1-1)^2$ and so $\overline{G}_{\mathbb{Z}/n\mathbb{Z}}$  is not Ramanujan. Assume $s\leq3$ in the sequel.

\smallskip
\noindent\emph{Case 3.1:} $s=3$. In this case, if $\alpha_1\geq3$, $\alpha_2\geq2$, $\alpha_3\geq2$ or $p_1\geq7$, or if $n=3^2p_2p_3$, $5p_2p_3$ or $5^2p_2p_3$, then $\varphi(n)^2/n = \prod_{i=1}^3 p_i^{\alpha_i-2} (p_i-1)^2 > \prod_{i=1}^3 p_i^{\alpha_i-1}(p_i-2) \geq 4(p_1-1)^2$ and so $\overline{G}_{\mathbb{Z}/n\mathbb{Z}}$  is not Ramanujan. Thus $\overline{G}_{\mathbb{Z}/n\mathbb{Z}}$ is not Ramanujan unless $n=3p_2p_3$. Moreover, if $n=3p_2p_3$ but $(p_2, p_3) \ne (5, 7)$, then $\varphi(n)^2/n = 2^2 (p_2-1)^2(p_3-1)^2/3p_2p_3 > (p_2-2)(p_3-2)\geq16$ and so $\overline{G}_{\mathbb{Z}/n\mathbb{Z}}$ is not Ramanujan; if $n=3\cdot5\cdot7$, then (\ref{CXRamanujan22}) is violated and again $\overline{G}_{\mathbb{Z}/n\mathbb{Z}}$ is not Ramanujan.

\smallskip
\noindent\emph{Case 3.2:} $s=2$. In this case, by Lemma \ref{Inequality} we have
$\varphi(n)^2/n > \varphi(n)/2 = p_1^{\alpha_1-1}p_2^{\alpha_2-1}(p_1-1)(p_2-1)/2$. Thus, if $\alpha_1\geq3$, $\alpha_2\geq3$ or $\alpha_1=\alpha_2=2$, then $\varphi(n)^2/n > 4(p_1-1)^2$ and so $\overline{G}_{\mathbb{Z}/n\mathbb{Z}}$ is not Ramanujan. In other words, $\overline{G}_{\mathbb{Z}/n\mathbb{Z}}$ is Ramanujan only when $n = p_1p_2$, $p_1p_2^2$ or $p_1^2p_2$.

If $n=p_1p_2$, then (\ref{CXRamanujan22}) is mounted to $(p_2 - 4)^2 \le 4p_1$. Therefore, if $n=3\cdot5$, $3\cdot7$ or $5\cdot7$, then $\overline{G}_{\mathbb{Z}/n\mathbb{Z}}$ is Ramanujan, and for $(p_1, p_2) \ne (3,5), (3,7), (5,7)$, $\overline{G}_{\mathbb{Z}/n\mathbb{Z}}$ is not Ramanujan.

If $n=p_1p_2^2$, then $\varphi(n)^2/n = (p_1-1)^2(p_2-1)^2/p_1 \geq 4(p_1-1)^2$. So $\overline{G}_{\mathbb{Z}/n\mathbb{Z}}$ is not Ramanujan.

If $n=p_1^2p_2$, then $\varphi(n)^2/n = (p_1-1)^2(p_2-1)^2/p_2 > (p_2-2)(p_1-1)^2$. Thus, if $p_2 \ge 7$, then $\varphi(n)^2/n \geq 4(p_1-1)^2$ and so $\overline{G}_{\mathbb{Z}/n\mathbb{Z}}$ is not Ramanujan. If $p_2=5$, then by (\ref{CXRamanujan22}) and again $\overline{G}_{\mathbb{Z}/n\mathbb{Z}}$ is not Ramanujan.
\qed\end{proof}

\section{Energy of the line graph of a unitary Cayley graph}
\label{sec:energy}

The \emph{iterated line graphs} of a graph $G$ are defined by $\mathcal {L}^1(G)=\mathcal {L}(G)$ and $\mathcal {L}^{i+1}(G)=\mathcal {L}(\mathcal {L}^{i}(G))$ for $i \ge 1$. It was proved in \cite{kn:Ramane05} that, if $G$ is an $r$-regular graph of order $n$ with $r\geq3$, then
$E(\mathcal {L}^{i+1}(G))=2nr(r-2)\prod_{j=0}^{i-1}\left(2^jr-2^{j+1}+2\right)$ for every $i\geq1$. However, there is no known closed formula for $E(\mathcal {L}(G))$ even when $G$ is regular, though $E(\mathcal {L}(G))$ has been computed for some special graphs such as caterpillars and certain combinations of generalized Bethe trees \cite{kn:Rojo11,kn:Rojo111}.

\begin{thm}\label{thmE}
Let $R$ be as in Assumption \ref{as:1}. Then
\be
\label{eq:energy}
E(\mathcal {L}(G_{R})) =
\left\{\begin{array}{rcl}
2^{s+1}(|R^\times|-1)^2,  & &   \text{if $2=|R_1|/m_1=\cdots=|R_s|/m_s$,}\\
                          & & \text{or $R=\underbrace{\mathbb{F}_2\times\cdots\times\mathbb{F}_2}_{s-1}\times\mathbb{F}_3$;}\\ [0.2cm]
2^{t+1} +2|R|\left(|R^\times|-2\right), & & \text{if $2=|R_1|/m_1=\cdots=|R_t|/m_t<|R_{t+1}|/m_{t+1}$}\\
                  & & \text{with $1\leq t<s$ and $R\ncong\underbrace{\mathbb{F}_2\times\cdots\times\mathbb{F}_2}_{s-1}\times\mathbb{F}_3$};\\ [0.2cm]
2|R|\left(|R^\times|-2\right), & & \text{if $3\leq|R_1|/m_1\leq\cdots\leq|R_s|/m_s$ and $R\ncong\mathbb{F}_3$}.\\
                                \end{array}\right.
\ee
\end{thm}

In the special case where $R = \mathbb{Z}/n\mathbb{Z}$, Theorem \ref{thmE} yields the following result.

\begin{cor}\label{corTE}
Let $n=p_1^{\alpha_1}p_2^{\alpha_2}\cdots p_s^{\alpha_s}$ be as in Corollary \ref{XnCor1}. Then
\be
\label{eq:CTE}
E(\mathcal {L}(G_{\mathbb{Z}/n\mathbb{Z}})) =
\left\{\begin{array}{rcl}
           4,               & & \text{if $n=3$;}\\ [0.2cm]
           8,               & & \text{if $n=6$;}\\ [0.2cm]
4\left(2^{\alpha_1-1}-1\right)^2,  & &   \text{if $n=2^{\alpha_1}$;}\\[0.2cm]
4 +2n\left(\left(\prod_{i=1}^sp_i^{\alpha_1-1}(p_i-1)\right)-2\right), & & \text{if $2=p_1$ and $n\ne6$};\\ [0.2cm]
2n\left(\left(\prod_{i=1}^sp_i^{\alpha_1-1}(p_i-1)\right)-2\right), & & \text{if $3\le p_1$ and $n\neq3$}.\\
                                \end{array}\right.
\ee
\end{cor}

A graph $G$ with $n$ vertices is called \emph{hyperenergetic} \cite{Gutman99} if $E(G) > 2(n-1)$. By Theorem \ref{thmE} we know exactly when $\mathcal {L}(G_{R})$ is hyperenergetic, as stated in the following corollary. (The fact that $\mathcal {L}(G_{R})$ is hyperenergetic when $|R^\times|\ge4$ can be also obtained from the following known result \cite{HG}: If $G$ has more than $2n-1$ edges, then $\mathcal{L}(G)$ is hyperenergetic.)

\begin{cor}\label{corhyper}
Let $R$ be as in Assumption \ref{as:1}. Then $\mathcal {L}(G_{R})$ is hyperenergetic if and only if one of the following holds:
\begin{itemize}
\item[\rm (a)] $|R^\times| \ge 4$;
\item[\rm (b)] $s = 1$ and $|R| = 2m \ge 8$;
\item[\rm (c)] $s \ge 2$, $2=|R_1|/m_1=\cdots=|R_s|/m_s$, and $|R^\times| \ge 2$.
\end{itemize}
\end{cor}

\begin{Tproof} \textbf{of Theorem \ref{thmE}.}
The proof consists of Lemmas \ref{lem:1}--\ref{lem:4} as follows.
\qed
\end{Tproof}

\begin{lem}\label{lem:1}
Let $R$ be a finite local ring with maximal ideal $M$ of order $m$. Then
\[E(\mathcal {L}(G_{R}))=\left\{\begin{array}{rcl}
                                   4\left(|R^\times|-1\right)^2,  &   &   \text{if $R/M\cong\mathbb{F}_2$, or $R\cong\mathbb{F}_3$;}\\ [0.2cm]
                                  2|R|\left(|R^\times|-2\right), &   & \text{otherwise.}
                                \end{array}\right.\]
\end{lem}
\begin{proof}
Let us begin with a few observations. Denote $q=|R|/m$. Then $|R^\times|=|R|-m=(q-1)m \geq 1$. Moreover, $|R^\times|=1$ if and only if $R\cong\mathbb{F}_2$, and $|R^\times|=2$ if and only if $R\cong\mathbb{F}_3$, $\mathbb{Z}_4$ or $\mathbb{Z}_2[X]/(X^2)$. Similarly, $|R^\times|-m=(q-2)m \geq 0$; $|R^\times|-m=0$ if and only if $R/M\cong\mathbb{F}_2$; $|R^\times|-m=1$ if and only if $R\cong\mathbb{F}_3$; and $|R^\times|-m=2$ if and only if $R\cong\mathbb{F}_4$.

If $R\cong\mathbb{F}_3$, then $\mathcal {L}(G_{R})\cong C_3$ and so $E(\mathcal {L}(G_{R})) = |2| + 2 \cdot |-1| = 4$.

If $R/M\cong\mathbb{F}_2$, then $|R|/m=2$ and $|R^\times|=m$. If $R\cong\mathbb{F}_2$, then $\mathcal {L}(G_{R})$ is an isolated vertex, which has energy $0$. If $R\ncong\mathbb{F}_2$, then $|R^\times|=m\geq2$, and by Corollary \ref{lineeigenvalue},
\begin{eqnarray*}
\Spec (\mathcal {L}(G_R))=\left(\begin{array}{ccc}
2|R^\times|-2   & |R^\times|-2 &-2\\
1              & 2(|R^\times|-1)& (|R^\times|-1)^2\\
\end{array}
\right)
\end{eqnarray*}
and so $E(\mathcal {L}(G_{R}))=4(|R^\times|-1)^2$.
In view of the computation above, this formula also applies when $R\cong\mathbb{F}_3$ or $\mathbb{F}_2$.

If $R\ncong\mathbb{F}_3$ and $R/M\ncong\mathbb{F}_2$, then $2|R^\times|-2\geq0$, $|R^\times|-m-2\geq0$ and $|R^\times|-2\geq0$. The proof is straightforward by using Corollary \ref{lineeigenvalue} again.
\qed\end{proof}

\begin{lem}\label{lem:2}
Let $R$ be as in Assumption \ref{as:1} with $|R_1|/m_1\geq3$ and $s \ge 2$. Then
\[E(\mathcal {L}(G_{R}))=2|R|\left(|R^\times|-2\right).\]
\end{lem}

\begin{proof}
By the definition of $\lambda_C$ and Assumption \ref{as:1}, for every $C\subseteq N$, the corresponding eigenvalue in (a) of Corollary \ref{lineeigenvalue} is
$$
\lambda_C+|R^\times|-2 \ge -\dfrac{|R^\times|}{|R^\times_1|/m_1}+|R^\times|-2= \left(m_1((|R_1|/m_1)-2)\prod_{i=2}^sm_i((|R_i|/m_i)-1)\right)-2 \ge 0.
$$
Hence
\begin{eqnarray*}
 \sum_{C\subseteq N} \big|\lambda_C+|R^\times|-2\big|\cdot \prod_{j\in C}\frac{|R_j^\times|}{m_j}
    &=& \sum_{C\subseteq N}\left((-1)^{|C|}|R^\times|+\left(|R^\times|-2\right)\cdot\prod_{j\in C}\frac{|R_j^\times|}{m_j}\right)\\
    &=& |R^\times|\cdot\sum_{C\subseteq N}(-1)^{|C|}+\left(|R^\times|-2\right)\cdot\sum_{C\subseteq N} \prod_{j\in C}\frac{|R_j^\times|}{m_j}\\
    &=&\left(|R^\times|-2\right)\prod_{i=1}^s\left(1+\dfrac{|R^\times_i|}{m_i}\right).
\end{eqnarray*}
Since $|R^\times|-2=\prod_{i=1}^s(|R_i|-m_i)-2=\prod_{i=1}^s m_i \left((|R_{i}|/m_{i})-1\right)-2>0$, by Corollary \ref{lineeigenvalue},
\begin{eqnarray*}
E(\mathcal {L}(G_{R}))&=&\left(|R^\times|-2\right)\prod_{i=1}^s\left(1+\dfrac{|R^\times_i|}{m_i}\right)\\
           &&+\left(|R^\times|-2\right)\cdot\left(|R|-\prod_{i=1}^s\left(1+\dfrac{|R^\times_i|}{m_i}\right)\right)+|R|\left(|R^\times|-2\right)\\
           &=&2|R|\left(|R^\times|-2\right).
\end{eqnarray*}
\qed\end{proof}

\begin{lem}\label{lem:3}
Let $R$ be as in Assumption \ref{as:1} with $|R_i|/m_i=2$, $1\leq i\leq s$, and $s \ge 2$. Then
\[E(\mathcal {L}(G_{R}))=2^{s+1}(|R^\times|-1)^2.\]
\end{lem}

\begin{proof}
Since $|R^\times|=\prod_{i=1}^s m_i \geq 1$, $|R^\times|=1$ if and only if $R_i\cong\mathbb{F}_2$ for $1\leq i\leq s$.

If $R_i\cong\mathbb{F}_2$ for $1\leq i\leq s$, then the spectrum of $\mathcal {L}(G_{R})$ is $0$ with multiplicity $2^{s-1}$. In this case, the energy of $\mathcal {L}(G_{R})$ is $0$.

If not every $R_i$ is $\mathbb{F}_2$, then $|R^\times|\geq2$ and by Corollary \ref{lineeigenvalue},
\begin{eqnarray*}
\Spec (\mathcal {L}(G_R))=\left(\begin{array}{ccc}
2|R^\times|-2            & |R^\times|-2          &   -2\\
2^{s-1}                  & |R|-2^s               & 2^{s-1}+|R|(|R^\times|-2)/2\\
\end{array}
\right).
\end{eqnarray*}
Since $|R|=2^s|R^\times|$, we have
$E(\mathcal {L}(G_{R})) = (2|R^\times|-2)2^{s-1} +(|R^\times|-2)\left(|R|-2^s\right)+2(2^{s-1}+|R|(|R^\times|-2)/2) = 2^{s+1}(|R^\times|-1)^2$.
\qed\end{proof}

\begin{lem}\label{lem:4}
Let $R$ be as in Assumption \ref{as:1} with $2=|R_1|/m_1=\cdots=|R_t|/m_t < |R_{t+1}|/m_{t+1}$ for some $t$ such that $1 \le t < s$. Then
\[E(\mathcal {L}(G_{R}))=\left\{\begin{array}{rcl}
                                   2^{s+1}(|R^\times|-1)^2,  &   &   \text{if $R=\underbrace{\mathbb{F}_2\times\cdots\times\mathbb{F}_2}_{s-1}\times\mathbb{F}_3$;}\\
                                  2^{t+1} +2|R|\left(|R^\times|-2\right), &   & \text{otherwise.}
                                \end{array}\right.\]
\end{lem}

\begin{proof} First, $|R^\times|\geq2$. Denote by $\lambda^a_{\sec}$ $(\neq-2)$ the second smallest eigenvalue in (a) of Corollary \ref{lineeigenvalue}. Then
\begin{eqnarray*}
  \lambda^a_{\sec} &=& |R^\times|-\dfrac{|R^\times|}{|R^\times_{t+1}|/m_{t+1}}-2\nonumber\\
                   &=& \left(|R_{t+1}|/m_{t+1}-2\right)\left(\prod_{i=1}^sm_i\right)\left(\prod_{i=t+2}^s(|R_i|/m_i-1)\right)-2\geq-1.
\end{eqnarray*}
Thus $\lambda^a_{\sec}=-1$ if and only if $s=t+1$, $|R_{t+1}|/m_{t+1}=3$ and $\prod_{i=1}^sm_i=1$. That is, $\lambda^a_{\sec}=-1$ if and only if $R=\underbrace{\mathbb{F}_2\times\cdots\times\mathbb{F}_2}_t\times\mathbb{F}_3$. In all other cases, we have $\lambda^a_{\sec}\geq0$.

If $R=\underbrace{\mathbb{F}_2\times\cdots\times\mathbb{F}_2}_t\times\mathbb{F}_3$, then by Corollary \ref{lineeigenvalue},
\begin{eqnarray*}
\Spec (\mathcal {L}(G_R))=\left(\begin{array}{cccc}
2           & 1          &   -1 &-2\\
2^{t-1}     & 2^t        & 2^t  & 2^{t-1}\\
\end{array}
\right).
\end{eqnarray*}
Since $|R^\times|=2$, we obtain
$$
E(\mathcal {L}(G_{R}))=2 \cdot 2^{t-1}+1 \cdot 2^t+1 \cdot 2^t+2 \cdot 2^{t-1}=2^{t+2}(|R^\times|-1)^2.
$$

If $R \ne \underbrace{\mathbb{F}_2\times\cdots\times\mathbb{F}_2}_t\times\mathbb{F}_3$, then $\lambda^a_{\sec}\geq0$ and so $-2$ is the unique negative eigenvalue. Since by Corollary \ref{lineeigenvalue} the multiplicity of $-2$ is $2^{t-1}$, we have
\begin{eqnarray*}
 \sum_{C\subseteq N} \left|\lambda_C+|R^\times|-2\right|\cdot\prod_{j\in C}\frac{|R_j^\times|}{m_j}
    &=& 2 \cdot 2 \cdot 2^{t-1}+\sum_{C\subseteq N}\left((-1)^{|C|}|R^\times|+\left(|R^\times|-2\right)\cdot\prod_{j\in C}\frac{|R_j^\times|}{m_j}\right)\\
    &=&2^{t+1} +|R^\times|\cdot\sum_{C\subseteq N}(-1)^{|C|}+\left(|R^\times|-2\right)\cdot\sum_{C\subseteq N}\left(\prod_{j\in C}\frac{|R_j^\times|}{m_j}\right)\\
    &=&2^{t+1} +\left(|R^\times|-2\right)\prod_{i=1}^s\left(1+\dfrac{|R^\times_i|}{m_i}\right).\\
\end{eqnarray*}
Therefore,
\begin{eqnarray*}
E(\mathcal {L}(G_{R}))&=&2^{t+1} +\left(|R^\times|-2\right)\prod_{i=1}^s\left(1+\dfrac{|R^\times_i|}{m_i}\right)\\
&&+\left(|R^\times|-2\right)\cdot\left(|R|-\prod_{i=1}^s\left(1+\dfrac{|R^\times_i|}{m_i}\right)\right)+|R|\left(|R^\times|-2\right)\\
           &=&2^{t+1} +2|R|\left(|R^\times|-2\right).
\end{eqnarray*}
\qed\end{proof}

\section{Spectral moments of unitary Cayley graphs and their line graphs}
\label{sec:moments}

\begin{thm}\label{thmSM}
Let $R$ be as in Assumption \ref{as:1}. Then, for any integer $k \ge 1$,
\begin{eqnarray}
s_k(G_R) & = & |R^\times|\prod_{i=1}^s\left(|R_i^\times|^{k-1}-(-m_i)^{k-1}\right) \label{thmE1}\\
s_k(\mathcal {L}(G_R)) & = & \left(\sum_{j=0}^k
\choose{k}{j} (|R^\times|-2)^{k-j}s_j(G_R)\right) - (-2)^{k-1}|R|(|R^\times|-2).\label{thmE2}
\end{eqnarray}
\end{thm}

To prove this we need the following lemma.

\begin{lem}
\label{SMrelation}
Let $G$ be an $r$-regular graph of order $n$. Then
$$
  s_k(\mathcal {L}(G)) = \left(\sum_{j=0}^k
  \choose{k}{j} (r-2)^{k-j}s_j(G)\right) - (-2)^{k-1}n(r-2).
$$
\end{lem}

\begin{proof}
Since $G$ is $r$-regular, the eigenvalues (see \cite[Theorem 2.4.1]{kn:Cvetkovic10}) of $\mathcal {L}(G)$ are $\lambda_i+r-2$, for $i=1,2,\ldots,n$, and $-2$ repeated $n(r-2)/2$ times, where $\lambda_1,\lambda_2,\ldots,\lambda_n$ are the eigenvalues of $G$. The result then follows from a straightforward computation.
\qed\end{proof}

As a consequence of Lemma \ref{tensorspectrum}, we have $s_k(G\otimes H) = s_k(G)\cdot s_k(H)$. In general, by induction, we see that the $k$-th spectral moment of the tensor product of a finite number of graphs is equal to the product of the $k$-th moments of the factor graphs.

\bigskip

\begin{Tproof}
\textbf{of Theorem \ref{thmSM}.}~~
By Lemma \ref{ringspectrum},
$$
s_k(G_{R_i}) = |R_i^\times|^k+(-m_i)^k\cdot\frac{|R_i^\times|}{m_i}=|R_i^\times|\cdot\left(|R_i^\times|^{k-1}-(-m_i)^{k-1}\right).
$$
Since $G_R=\bigotimes_{i=1}^sG_{R_i}$ as mentioned in \textsection \ref{sec:prel}, from the discussion above we obtain
\begin{eqnarray*}
  s_k(G_R) &=&\prod_{s=1}^ss_k(G_{R_i})=\prod_{i=1}^s|R_i^\times|\left(|R_i^\times|^{k-1}-(-m_i)^{k-1}\right)\nonumber\\
           &=&|R^\times|\prod_{i=1}^s\left(|R_i^\times|^{k-1}-(-m_i)^{k-1}\right),
\end{eqnarray*}
which is exactly (\ref{thmE1}).

Since $G_R$ is $|R^\times|$-regular with order $|R|$, (\ref{thmE2}) follows from Lemma \ref{SMrelation}  and (\ref{thmE1}).
\qed\end{Tproof}

Denote by $n_3(G)$ the number of triangles in a graph $G$. Since $s_3(G)=6n_3(G)$ \cite{kn:Cvetkovic10}, Theorem \ref{thmSM} implies the following formulae.

\begin{cor}
Let $R$ be as in Assumption \ref{as:1}. Then
\begin{eqnarray*}
  n_3(G_R) &=& \frac{1}{6}|R^\times||R|\prod_{i=1}^s\left(|R_i^\times|-m_i\right) \\
  n_3(\mathcal {L}(G_R)) &=&  \frac{1}{6}|R^\times||R|\left(\prod_{i=1}^s\left(|R_i^\times|-m_i\right)+\left(|R^\times|-1\right)\left(|R^\times|-2\right)\right).
\end{eqnarray*}
\end{cor}

Denote by $n_4(G)$ the number of quadrangles ($4$-cycles) in $G$. It is well known \cite{kn:Cvetkovic10} that, if $G$ is a graph with $n$ vertices, $m$ edges and degree sequence $(d_1,d_2,\ldots,d_n)$, then $s_4(G) = 2m+4\sum_{j=1}^n \choose{d_j}{2}+8n_4(G)$. This and Theorem \ref{thmSM} together imply the following formulae.

\begin{cor}
Let $R$ be as in Assumption \ref{as:1}. Then
\begin{eqnarray*}
  n_4(G_R) &=& \frac{1}{8}|R^\times||R|\left(1-2|R^\times|+\prod_{i=1}^s\left(|R_i^\times|^2-|R_i^\times|m_i+m_i^2\right)\right) \\
  n_4(\mathcal {L}(G_R)) &=& \frac{1}{8}|R^\times||R|\left(|R^\times|(|R^\times|-3)^2-5+4(|R^\times|-2)\prod_{i=1}^s\left(|R_i^\times|-m_i\right)\right.\nonumber\\
  &&\left.+\prod_{i=1}^s\left(|R_i^\times|^2-|R_i^\times|m_i+m_i^2\right)\right).
\end{eqnarray*}
\end{cor}

\medskip

\noindent \textbf{Acknowledgements}~~X. Liu is supported by MIFRS and MIRS of the University of Melbourne. S. Zhou is supported by a Future Fellowship (FT110100629) of the Australian Research Council.

\end{document}